\documentclass[12pt,reqno]{amsart}
\usepackage{a4wide}
 \usepackage[misc]{ifsym} 
\numberwithin{equation}{section}

\usepackage{amsmath}

\usepackage{color}
\usepackage{txfonts}
\usepackage{amsfonts}
\usepackage{amsmath,amsthm,amssymb,amscd}
\usepackage{latexsym}
\usepackage{hyperref}
\usepackage[numbers,sort&compress]{natbib}
\usepackage{hypernat}
\allowdisplaybreaks[4]
\numberwithin{equation}{section}

\newtheorem{theorem}{Theorem}[section]
\newtheorem{proposition}[theorem]{Proposition}

\newtheorem{lemma}[theorem]{Lemma}

\theoremstyle{definition}

\newcommand{\va}{\varepsilon}

\def\r{\mathbb{R}}

\begin{document}

\title[	Solutions for singularly perturbed Schr\"{o}dinger equations]
{Infinitely many new solutions for singularly perturbed Schr\"{o}dinger equations}

\author{ Benniao Li$^{1}$,\ \ Wei Long$^{1,\text{\Letter}}$,\ \  Jianfu Yang$^{1}$}



\thanks{ $^{1}$School of Mathematics
and Statistics, Jiangxi Normal University, Nanchang,
Jiangxi 330022, P. R. China }

\thanks{benniao\_li@jxnu.edu.cn(B. Li)}
\thanks{ lwhope@jxnu.edu.cn(\Letter W. Long)}
\thanks{jfyang200749@sina.com(J. Yang)}

\begin{abstract}

This paper deals with the existence of solutions for the following perturbed Schr\"{o}dinger equation
\begin{equation*}
-\varepsilon^{2} \Delta u + V(x)u= |u|^{p-2}u, \, \, \text{ in } \, \, \r^{N}, 
\end{equation*}
where $\varepsilon$ is a parameter, $N \geq 3$, $p \in (2, \frac{2N}{N-2})$, and $V(x)$ is a potential function in $\r^{N}$. We demonstrate an interesting ``dichotomy'' phenomenon for concentrating solutions of the above Schr\"{o}dinger equation. More specifically, we construct infinitely many new solutions with peaks locating both in the bounded domain and near infinity, which fulfills the profile of the concentration compactness. Moreover, this approach can be extended to solve other related problems.

 {\bf Key words}: Schr\"{o}dinger equations; non-degeneracy; dichotomy.

{\bf AMS Subject Classifications:} 35J50, 35J60
\end{abstract}

\maketitle

\section{introduction}

In this paper, we consider the following nonlinear Schr\"odinger problem
\begin{equation}\label{main}
-\varepsilon^{2} \Delta u + V(x)u= |u|^{p-2}u, \, \, \text{ in } \, \, \r^{N},
\end{equation}
where $\varepsilon$ is a parameter, $N\geq 3$, $p \in (2, \frac{2N}{N-2})$. Problem \eqref{main} arises in the study of standing-wave solutions for the time-dependent Schr\"odinger equation
\begin{equation}\label{main1a} i \varepsilon \frac{\partial \psi}{\partial t} = - \varepsilon^{2} \Delta \psi + \bigl( V(x) +E \bigr) \psi - |\psi|^{p-2}\psi,  (x, t) \in \r^{N} \times \r_{+}.
\end{equation}
A standing-wave solution of  problem \eqref{main1a} is a solution of the form \[\psi(x, t) = exp(-{ i E t}/{\varepsilon})u(x),\] where $i$ is the imaginary unit, $\varepsilon$ is the Planck constant,  and then $u(x)$ is a solution of \eqref{main}. Problem \eqref{main1a} appears in nonlinear optics and quantum physics. It describes the transition from quantum to classical mechanics as $\varepsilon \to 0$, which attracts a lot of researches.

 When $\varepsilon $ is fixed, we may assume $\varepsilon=1$, the following problem
 \begin{equation}\label{main1}
- \Delta u + V(x)u= |u|^{p-2}u, \, \, \text{ in } \, \, \r^{N}
\end{equation}
has been widely studied in recent decades. By the variational method, Ding-Ni proved in \cite{DN} that the problem \eqref{main1} has a nontrivial radial solution if $V(x)$ is radial. While for non-radial $V(x)$, it was shown in \cite{Ra} the existence of the ground state solution for \eqref{main1} provided that
 \[\lim\inf_{|x| \to \infty } V(x) > \inf_{x\in \r^{N}} V(x) \]
by the variational method and the concentration-compactness principle. In \cite{CDS}, Cerami et al obtained infinitely many solutions for the problem \eqref{main1}. Assuming $V(x)$ satisfies certain conditions, they proved, among other things, by the concentration compactness principle \cite{Li1, Li2}  that
 a Palais-Smale sequence $\{u_{n}\}$ can be decomposed as
\[u_{n} = u_{0} + \sum_{i=1}^{\kappa} {U}(\cdot- x_{n}^{i}) + o(1)\]
where $u_{n} $ weakly converges to $u_{0}$, $|x_{n}^{i} - x_{n}^{j}| \to +\infty$ and $|x_{n}^{i}| \to +\infty$ if $i \neq j$.
In order to verify that $\{u_n\}$ satisfies the Palais-Smale condition, one generally kills bumps at infinity, which implies the sequence is strongly convergent.  Solutions obtained in this way vanish at infinity, see for instance \cite{Ba1, C, Ra}.

 When $\varepsilon$ is sufficiently small, Wang showed in \cite{w} that the least energy solutions for \eqref{main} concentrate near the global minima of $V(x)$ as $\varepsilon \to 0$, see \cite{wz} for general cases. On the other hand, a solution of \eqref{main} concentrating on non-degenerate critical points of $V(x)$ was first constructed in \cite{FW} by the Lyapunov-Schmidt reduction method for sufficiently small $\varepsilon >0$ and $N=1$. Later on, by the same approach, Oh \cite{o1, o2} extended  the result to higher-dimensional cases with subcritical nonlinearity. Since then, the reduction method has become a powerful device for constructing solutions. In particular,   multi-peak solutions for \eqref{main}  are constructed in \cite{df1} if $V(x)$ is locally H\"older continuous in $\r^{N}$ and there exist $k$ disjoint bounded regions $\Omega_{i} \,(\,i =1, 2, \cdots, k)$ in $\r^{N}$ such that
 \[ \inf_{x\in \partial \Omega_{i}} V(x) > \inf_{x\in \Omega_{i}} V(x) > 0. \]
 Further results can be found in \cite{CNY, df, L}. From these results, we see that critical points of $V(x)$ located in bounded domains play an important role in constructing solutions concentrated in bounded domains.

 In contrast to works in \cite{CNY, df, df1, L}, it was considered in  \cite{WY} the existence of infinitely many solutions concentrating near infinity. Suppose that $V$ satisfies
 \begin{equation}\label{main1b}
 V(|x|) =V(r) = V_{\infty} + \frac{a}{r^{\alpha}} +O(r^{-\alpha - \delta}), r \to \infty
 \end{equation}
with $a, \delta> 0$ and $\alpha >1$. In order to construct infinitely many solutions concentrating near infinity,  a new method was developed in  \cite{WY}, where the number of peaks is treated as the parameter.
Condition \eqref{main1b} can be regarded as the critical points of $V(x)$ near infinity. Results for non-radial functions $V(x)$ can be found in \cite{DWY}.

In previous works,  peaks of solutions for \eqref{main} are either located in bounded domains or near infinity.  In this paper, we are interested in construct solutions of \eqref{main} concentrating both on a bounded domain and near infinity. Inspired by \cite{GMPY} such solutions will be constructed by assembling two type solutions, one concentrating in bounded domain and the other at infinity.

We assume in this paper that $V(x)$ satisfies the following conditions.\\
 $(V_1).$ $V(x)$ is radially symmetric with respect to $x_{0} \in \mathbb{R}^{N}$ and $x_{0}$ is a non-degenerate critical point of $V(x)$;\\
$(V_2^\pm).$  $V(|x-x_{0}|)=V(r) = V_{\infty} \pm \frac{a}{r^{m}} + O(r^{-(m + \delta)})$ as $r \to +\infty$ with $m> 1$ .\\

Without loss of generality, we may assume that $x_{0} = 0$ and $V_{\infty} = 1$ in the rest paper. It is well-known that  the problem
\begin{equation}\label{eq:1.3}
\left\{\begin{array}{lll} & -\Delta u + u = u^{p-1}, u > 0,\\
&  u(0) = \max_{x \in \r^N} u(x).
\end{array}
\right.
\end{equation}
has a unique positive solution $U(x)$, which  is, by \cite{GNN, K},  radially symmetric, $U'(r)< 0$ and ${U}$ satisfies
\[\lim_{r\to \infty} {U}(r) e^{r}r^{\frac{N-1}{2}} = C > 0, \, \, \, \lim_{r \to \infty } \frac{{U}^{'}}{{U}} = -1.\]
Moreover, ${U}$ is non-degenerate, that is,  the kernel of $-\Delta \varphi + \varphi - (p-1) {U}^{p-2}\varphi $ equals to
\[span\Big\{\frac{\partial {U}}{\partial x_{i}}, i = 1, 2, \cdots, N\Bigr\}.
\]
 Denote $U_{\varepsilon}(x) = {U}(x/\varepsilon)$, then $U_{\varepsilon}(x)$ is a solution of the problem
\[-\varepsilon^{2} \Delta u +  u = |u|^{p-2}u, u > 0 \, \, u(0) = \max _{x \in \r^{N} } u(x),\]
and  $U_{\varepsilon}$ is non-degenerate as well, i.e.  if
\begin{equation}\label{non}
-\varepsilon^{2} \Delta\varphi + \varphi  - (p-1)U_{\varepsilon}^{p-2} \varphi = 0, \end{equation}
then
\[\varphi = \sum_{i=1}^{N} a_{i} \frac{\partial U_{\varepsilon}}{\partial x_{i}},\]
for some constants $a_{i}$, $i =1, 2, \cdots, N$.

At first, through the result from \cite{o1, o2} and the moving plane method, we can get the following theorem directly:
\begin{theorem}\label{thm0} Assume that $V(x)$ satisfies 	Condition $(V_{1})$. Then there exists $\varepsilon_{0} > 0$, such that for any $\varepsilon \in (0, \varepsilon_{0})$, there is a radial solution $u_{\varepsilon}$ of \eqref{main} in the form
\[u_{\varepsilon}(x) = \Bigl( V(0)\Bigr)^{\frac{1}{p-2}} {U}\bigl(\frac{\sqrt{V(0)}}{\varepsilon} x\bigr) + \omega_{\varepsilon }(x) \]
with the radial error term $\|\omega_{\varepsilon }(x)\|_{H^{1}(\mathbb{R}^{N})}= o(\varepsilon^{\frac{N}{2}}) $.
\end{theorem}

For any integer $ k> 0$, let
\[\xi_{j} = \bigl( r \cos \frac{2(j-1)\pi}{k}, r\sin\frac{2(j-1)\pi}{k}, {\bf {0}}\bigr) \in \r^{2} \times \r^{N-2}, j =1, 2, \cdots, k,\]
where ${\bf {0}}$ is the zero vector in $\r^{N-2}$. Denote
 \[ U_{\varepsilon, \xi_{j}}(x) = U_{\varepsilon}(x - \xi_{j}).\]

 Our main results are stated as follows.
\begin{theorem}\label{thm1} Let $u_{\varepsilon}$ be a solution in Theorem \ref{thm0} with a sufficiently small constant  $\varepsilon > 0$. Suppose  that $V(x)$ satisfies $(V_1)$ and $(V_2^+)$.  Then, there exists an integer $k_{\varepsilon }>0$ depending on $\varepsilon$, such that for  $k > k_{\varepsilon}$, \eqref{main} has infinitely many solutions of the form
\begin{equation}\label{form} u_{\varepsilon, k}(x) = u_{\varepsilon}(x) + \sum_{j=1}^{k} U_{\varepsilon, \xi_{j}}(x)+o_{k}(1)\end{equation}
with $o_{k}(1) \to 0 $ as $k \to \infty$, where
\[r_{k} \in S_{k} :=\Bigl[\bigl(\frac{ \varepsilon m}{2\pi } -\theta\bigr) k \ln k,\,\, \bigl(\frac{\varepsilon m}{2 \pi } +\theta\bigr) k \ln k\Bigr],  \text{ for small }  \theta> 0 .\]
\end{theorem}

We remark that solutions of \eqref{main} given by \eqref{form} are dichotomy in the sense that  local maximal points of  $u_{\varepsilon}$ locate in a bounded domain, while that of
$U_{\varepsilon, \xi_{j}}(x)$ moves to infinity.

The proof of Theorem \ref{thm1} is to use the developed Lyapunov-Schmidt reduction method to construct solutions in the form
\begin{equation}\label{s}u_{\varepsilon,k} \approx u_{\varepsilon } + \sum_{j=1}^{k} U_{\varepsilon, \xi_{j}}(x). \end{equation}
We use $u_{\varepsilon}$ as a part of approximate solutions, which concentrates near the critical point of $V(x)$. The non-degeneracy of the linearized operator at $u_{\varepsilon}$ corresponding to  \eqref{main} is shown in \cite{T}.  Once $u_{\varepsilon}$ is  chosen, then the parameter $\varepsilon$ is fixed.  There is no more parameter as  the perturbation in \eqref{main}. We will use the number  of peaks $k$ as the parameter to construct the multi-peak solutions at infinity. A key point  in the estimate of the energy functional, among other things, is to find the balance between the term $\sum_{j=1}^{k} \int_{\mathbb{R}^{N}}(V(x)-1)U^{2}_{\varepsilon, \xi_{j}}dx$ arising from the potential function $V(x)$ and $\sum_{i\neq j} \int_{\mathbb{R}^{N}}U^{p}_{\varepsilon, \xi_{j}}U_{\varepsilon, \xi_{i}}dx$ coming from the interaction among the peaks. On the other hand, as we may see  in
 \cite{WY, WY1} that the number  of peaks $k \to \infty$ causes new difficulties since there exists the interaction among infinitely many peaks, while the concentrating location is dependent on the number of peaks, which is quite different from the problem with fixed number of peaks. Hence, the estimates between each of the two peaks should be more precise so that the infinite sum of interactions makes the proof work. Furthermore, although the process of the reduction method is standard, it is complicated to prove the invertibility of the linearized operator at the approximate solutions $u_{\varepsilon} + \sum_{j=1}^{k} U_{\varepsilon, \xi_{j}}(x)$ with two different types of concentrations.

If $V(x)$ satisfies $(V_1)$ and $(V_2^{-})$, by the same argument, we can construct infinitely many sign-changing solutions  for \eqref{main} in the form
\[u_{\varepsilon, k}(x) = u_{\varepsilon}(x) - \sum_{j=1}^{k} U_{\varepsilon, \xi_{j}}(x)+o_{k}(1).\]
Such a solution has positive peaks located in a bounded domain, and negative ones moving to infinity. Moreover, in this case, we can construct another type  sign-changing multi-bump solutions for \eqref{main},  which, besides peaks located in a bounded domain, has $k$ positive local maxima  and $k$ negative local minima moving to infinity.
Precisely,  for any integer $k$, let
\[\tilde \xi_{j} = \bigl(  r \cos \frac{(j-1)\pi}{k},  r\sin\frac{(j-1)\pi}{k}, {\bf {0}}\bigr) \in \r^{2} \times \r^{N-2}, j =1, 2, \cdots, 2k,\]
where ${\bf {0}}$ is the zero vector in $\r^{N-2}$. We have the following result.

\begin{theorem}\label{thm2} Assume that $V(x)$ satisfies $(V_1)$ and $(V_2^{-})$. Let $u_{\varepsilon}$ be a solution Theorem \ref{thm0} with sufficiently constant small $\varepsilon> 0$. Then, there exists an integer $\tilde k_{\varepsilon} > 0$ depending on $\varepsilon$ such that for $k > \tilde k_{\varepsilon }$, \eqref{main} has infinitely many sign-changing solutions of the form
\begin{equation}\label{form a}
u_{\varepsilon, k}(x) = u_{\varepsilon}(x) + \sum_{j=1}^{2k}(-1)^{j} U_{\varepsilon, \tilde \xi_{j}}(x)+\tilde{o}_{k}(1)
\end{equation}
with $\tilde{o}_{k}(1) \to 0 $ as $k \to \infty$, where
\[\tilde{r}_{k} \in \tilde{S}_{k} :=\Bigl[\bigl(\frac{\varepsilon m}{ \pi } -\theta\bigr) k \ln k,\,\, \bigl(\frac{\varepsilon m}{ \pi }+\theta\bigr) k \ln k\Bigr],  \text{ for small }  \theta> 0 .
\]
\end{theorem}
In order to prove Theorem \ref{thm2}  by the reduction method, we need to balance  the positive term $\int_{\mathbb{R}^{N}}U^{p}_{\varepsilon, \tilde{\xi}_{j}}U_{\varepsilon, \tilde{\xi}_{j+1}}dx$ in the corresponding functional  with the linear term $\int_{\mathbb{R}^{N}} (V(x)-1)U^{2}_{\varepsilon, \tilde{\xi}_{j}}$. To make it, $V(x)$ should satisfy $(V_{2}^{-})$, instead of $(V_{2}^{+})$.

This paper is organized as follows. In Section~2, we will introduce the finite-dimensional reduction method. The main theorems are proved  in Section~3.
\bigskip

\section{Finite-dimensional reduction}
\bigskip

In this section, we carry out the Lyapunov-Schmidt reduction procedure, which is generally based on the non-degeneracy of the related solution. Such a result fitting our approach was proved in \cite{T}, we state it as follows. Let $u_{\varepsilon}$ be the solution of \eqref{main} obtained in Theorem \ref{thm0}.
\begin{lemma}\label{non}\cite{T}
Assume that $V(x)$ satisfies $(V_{1})$, then $u_{\varepsilon}$  is non-degenerate, that is, if  $- \va^{2} \Delta \psi +V(x)\psi - (p-1)u_{\varepsilon}^{p-2} \psi = 0$ , then $\psi= 0$.
\end{lemma}

\bigskip

Let
\[\begin{split}H_{s} = \Bigl\{ & u \in H^{1}(\r^{N}), \, u(x) \text{ is even in } x_{i}, \,i =2, 3, \cdots N, \\
& u(r\cos \theta, r\sin \theta, x'') = u(r\cos (\theta + \frac{2\pi j}{k}), r\sin (\theta + \frac{2\pi j}{k}), x'' ), j=1, 2, \cdots, k  \Bigr\},
\end{split}\]
where the norm of $H^{1}(\r^{N})$ is induced by the inner product
\[ \langle  u, v \rangle = \int_{\r^{N}} \bigl(\varepsilon^{2} \nabla u \nabla v + V(x)u v \bigr) \, dx, \, \, \, u , v \in H^{1}(\r^{N}). \]
Define
\[E_{k} := \Bigl\{ v \in H_{s} : \sum_{j=1}^{k} \int_{\r^{N}} U_{\varepsilon, \xi_{j}}^{p-2}Z_{j}v \, dx = 0 \, \,  \Bigr\},\]
where
\[Z_{j}= \frac{\partial U_{\varepsilon, \xi_{j}}}{\partial r}, \, j=1, 2, \cdots, k.  \]
Solutions of \eqref{main} will be found as critical points of the associated functional
\[ I(u) = \frac12 \int_{\r^{N}} \Bigl(\varepsilon^{2} |\nabla u|^{2} + V(x) u^{2} \Bigr)\, dx  - \frac{1}{p} \int_{\r^{N}} |u|^{p}\, dx\]
of \eqref{main}. Particularly, we look for a solution $u$ of  \eqref{main} in the form
\[u(x) = u_{\varepsilon} + \sum_{j=1}^{k} U_{\varepsilon, \xi_{j}} + \omega_{k}: = W_{\varepsilon, k} + \omega_{k}. \]
Expanding the functional $I$, we require that $\omega_{k}$ satisfies the equation
\begin{equation}\label{eq:1-1-1}\left\{\begin{array}{lll}
L_{k} \omega =  l_{k} + R_{k}(\omega),\\
\omega \in H^{1}(\r^{N}),
\end{array}
\right.
\end{equation}
where $L_{k}$ is defined by

\begin{equation}\label{eq:1-1-2}
\langle L_{k}\omega, \varphi \rangle = \int_{\r^{N}} \bigl(\varepsilon^{2} \nabla \omega \nabla \varphi  +V(x) \omega \varphi- (p-1) W_{\varepsilon, k}^{p-2} \omega \varphi\bigr)\, dx, \, \, \forall \varphi \in H^{1}(\r^{N}),
\end{equation}
which is a bounded linear operator in $H^{1}(\r^{N})$, and $l_{k} \in H^{1}(\r^{N})$ satisfies
\begin{equation}\label{eq:1-1-3}
\langle l_{k}, \varphi \rangle =  \sum_{j=1}^{k}\int_{\r^{N}} \bigl(1 - V(x)\bigr) U_{\varepsilon, \xi_{j}} \varphi \, dx + \int_{\r^{N}} \Bigl( W_{\varepsilon, k}^{p-1} - u_{\varepsilon}^{p-1} -\sum_{j=1}^{k} U_{\varepsilon, \xi_{j}}^{p-1} \Bigr) \varphi\, dx,
\end{equation}
as well as  that $R_{k}(\omega) \in H^{1}(\r^{N})$ given by
\begin{equation}\label{eq:1-1-4}
\langle R_{k}(\omega) , \varphi \rangle = \int_{\r^{N}} \Bigl( \bigl( W_{\varepsilon, k}+ \omega  \bigr)^{p-1} -  W_{\varepsilon, k}^{p-1} - (p-1) W_{\varepsilon, k}^{p-2} \omega  \Bigr)\varphi \, dx.
\end{equation}

In order to solve \eqref{eq:1-1-1} and look for a solution in $H_{s}$,  we define the projection $P_{k}$ from $H_{s}$ to $E_{k}$ by
\[P_{k}u= u - a_{k} \sum_{j=1}^{k} Z_{j},  \]
where $a_{k}$ is determined by $\langle P_{k}u, \sum_{j=1}^{k}Z_{j}\rangle = 0 $. It can be proved that the equation  is solvable, so we can find $a_{k}$.

Now, we show that the operator $P_{k}L_{k}$ is invertible in $E_{k}$. Before starting to prove, we give some notations. Let $\varepsilon_0>0$ be as in Theorem \ref{thm0} and set
\[ \Omega_{j} = \Bigl\{x = (x', x'') \in \r^{2} \times \r^{N-2} : \langle \frac{(x', 0)}{|x'|},  \frac{\xi_{j}}{|\xi_{j}|} \rangle \geq \cos\frac{\pi}{k} \Bigr\}, \, \, j=1, 2, \cdots, k.\]
 We have

\begin{lemma}\label{lm1} For any $0<\varepsilon< \varepsilon_0$  and $k$ sufficiently large, there exists a positive constant $\rho>0$, such that $r \in S_{k}= \Bigl[\bigl(\frac{\varepsilon m}{ 2 \pi } -\theta\bigr) k \ln k, \bigl(\frac{\varepsilon m}{2  \pi } +\theta\bigr) k \ln k\Bigr] $,
\[\| P_{k}L_{k}u\| \geq \rho \|u\|, \, \, u \in E_{k}.\]
\end{lemma}

\begin{proof} We prove by contradiction. Assume on the contrary that there exist $r_{k}\in S_{k}$ and $v_{k} \in E_{k}$ as $k \to \infty$, such that
\[\| P_{k}L_{k} v_{k}\| = o_{k}(1) \|v_{k}\|. \]
Thus, for any $\psi \in E_{k}$,
\begin{equation}\label{eq:1.1}
\langle L_{k}v_{k}, \psi\rangle  = \langle P_{k} L_{k}v_{k}, \psi \rangle  = o_{k}(1) \|v_{k}\|\cdot \|\psi\|.
\end{equation}
 By the linearity of \eqref{eq:1.1}, we may assume that $\|v_{k}\|^{2}=k $.  Therefore, for $\psi \in E_k$,
\begin{equation}\label{eq:1.1-1}\begin{split}
o(\sqrt{k}) \|\psi\| = &  \int_{\r^{N}}\bigl( \varepsilon^{2} \nabla v_{k}\nabla \psi + V(x) v_{k}\psi - (p-1) W_{\varepsilon, k}^{p-2} v_{k} \psi \bigr) \, dx\\
= & \int_{\r^{N}}\Bigl( \varepsilon^{2} \nabla v_{k}\nabla \psi + V(x) v_{k}\psi - (p-1) (\sum_{j=1}^{k}  U_{\varepsilon, \xi_{j}})^{p-2} v_{k} \psi \Bigr) \, dx - (p-1) \int_{\r^N} u_\varepsilon^{p-2} v_k \psi \, dx \\
& + O\Bigl(\sum_{j=1}^k \int_{\r^N} u_\varepsilon^{\frac{p-2}{2}} U_{\varepsilon, \xi_{j}}^{\frac{p-2}{2}} v_k \psi \, dx\Bigr).
\end{split}
\end{equation}

In particular,
\begin{equation}\label{eq:1.1-5}
 \int_{\r^{N}}\bigl( \varepsilon^{2} |\nabla v_{k}|^{2} + V(x) v_{k}^{2} - (p-1) W_{\varepsilon, k}^{p-2} v_{k}^{2} \bigr) dx =  \langle L_{k} v_{k}, v_{k} \rangle =   \langle P_k L_{k} v_{k}, v_{k} \rangle =o(k).
\end{equation}

Since $\|v_{k}\|^{2} = k$, by symmetry we have
\[\int_{\Omega_{j}\setminus B_{R_{0}}(0)} \Bigl( \varepsilon^{2} |\nabla v_{k}|^{2} + V(x) v_{k}^{2} \Bigr)dx \leq 1. \]
Denote $\bar v_{k} = v_{k}(x + \xi_{1})$. Since $r_{k} \in S_{k}$, we get
\[|\xi_{2} - \xi_{1}| = r\sin \frac{\pi}{k} \geq \frac{2m-N+1}{16}\ln k. \]
Thus, for sufficiently large $k> 0$, there exists $R> 0$ such that $B_{R}(\xi_{1}) \subset \Omega_{1}\setminus B_{R_{0}}(0)$ and
\[\int_{B_{R}(0)}  \Bigl( \varepsilon^{2} |\nabla \bar v_{k}|^{2} + V(x)\bar v_{k}^{2} \Bigr)dx \leq 1. \]
So, up to subsequence, we may assume that there is $ v \in H^{1}(\r^{N})$ such that for $k \to \infty$,
\begin{equation}\label{weak1} \bar v_{k} \rightharpoonup v\, \,   \text{ in } \, H^{1}(\r^{N})  \end{equation}
and
\[\bar v_{k} \to v \, \, \text{ in } L_{loc}^{2}(\r^{N}). \]
The fact $v_{k} \in E_{k}$ implies
 \[\int_{\r^{N}} U_{\varepsilon}^{p-2} \frac{\partial U_{\varepsilon}}{\partial x_{1}} \bar v_{k} \, dx = 0. \]
By the decay of $U_{\varepsilon}$ at infinity,  for any $\sigma> 0$, there exists $R> 0$ such that
 \begin{equation}\label{weak4}
 \Bigl| \int_{\r^{N} \setminus B_{R}(0)} U_{\varepsilon}^{p-2} \frac{\partial U_{\varepsilon}}{\partial x_{1}} v \, dx  \Bigr|  < \frac{\sigma}{3} \text { and } \Bigl| \int_{\r^{N} \setminus B_{R}(0)} U_{\varepsilon}^{2p-4} \Bigl(  \frac{\partial U_{\varepsilon}}{\partial x_{1}} \Bigr)^{2} \, dx  \Bigr|  < \frac{\sigma^{2}}{36}.
 \end{equation}
Since $\bar v_{k}$ is bounded and converges to $v$ in $L_{loc}^{2}(\r^{N})$, then there exists $k_{0}$ such that for $k > k_{0}$, it holds
 \begin{equation}\label{weak3}
 \Bigl| \int_{B_{R}(0)} U_{\varepsilon}^{p-2} \frac{\partial U_{\varepsilon}}{\partial x_{1}} \bigl( \bar v_{k}- v\bigr) \, dx  \Bigr| < \frac{\sigma}{3}.
 \end{equation}
For $r \in S_{k}$,  there exists small $ \sigma_{1}> 0$, such that
  \[ \|\bar v_{k}\|_{H^{1}(B_{\sigma_{1} \ln k }(0))} \leq 1. \]
By \eqref{weak4} and $\|\bar v_{k}\| = \sqrt{k}$, we obtain
 \begin{equation}\label{weak5}\begin{split}
& \Bigl| \int_{\r^{N} \setminus B_{R}(0)} U_{\varepsilon}^{p-2} \frac{\partial U_{\varepsilon}}{\partial x_{1}}  \bar v_{k} \, dx \Bigr| \\
\leq  &  \Bigl| \int_{\r^{N} \setminus B_{\sigma_{1} \ln k }(0)} U_{\varepsilon}^{p-2} \frac{\partial U_{\varepsilon}}{\partial x_{1}}  \bar v_{k} \, dx \Bigr| +  \Bigl| \int_{B_{\sigma_{1} \ln k}(0)  \setminus B_{R}(0)} U_{\varepsilon}^{p-2} \frac{\partial U_{\varepsilon}}{\partial x_{1}}  \bar v_{k} \, dx \Bigr| \\
\leq & \frac{C_{\varepsilon}}{k^{p-1-\theta -\frac{1}{2}}} +\frac{\sigma}{6}.
\end{split}
 \end{equation}
 For $\bar k_{0}> k_{0} $ large and  $k > \bar k_{0}$, there holds
  \begin{equation}\label{weak6}\begin{split}
 \Bigl| \int_{\r^{N} \setminus B_{R}(0)} U_{\varepsilon}^{p-2} \frac{\partial U_{\varepsilon}}{\partial x_{1}}  \bar v_{k} \, dx \Bigr| <  \frac{\sigma}{3}.
\end{split}
 \end{equation}
From \eqref{weak4}, \eqref{weak3} and \eqref{weak6}, we deduce that for $k > \bar k_{0}$,
\begin{equation}\label{weak2}\begin{split}
& \Bigl| \int_{\r^{N}} U_{\varepsilon}^{p-2} \frac{\partial U_{\varepsilon}}{\partial x_{1}}  v \, dx  \Bigr|  \\
 = & \Bigl| \int_{\r^{N}} U_{\varepsilon}^{p-2} \frac{\partial U_{\varepsilon}}{\partial x_{1}} \bigl( \bar v_{k}- v\bigr) \, dx  \Bigr|  \\
\leq &  \Bigl| \int_{B_{R}(0)} U_{\varepsilon}^{p-2} \frac{\partial U_{\varepsilon}}{\partial x_{1}} \bigl( \bar v_{k}- v\bigr) \, dx  \Bigr|  +  \Bigl| \int_{\r^{N} \setminus B_{R}(0)} U_{\varepsilon}^{p-2} \frac{\partial U_{\varepsilon}}{\partial x_{1}} v \, dx  \Bigr|   +  \Bigl| \int_{\r^{N} \setminus B_{R}(0)} U_{\varepsilon}^{p-2} \frac{\partial U_{\varepsilon}}{\partial x_{1}}  \bar v_{k} \, dx \Bigr|  \\
 < & \sigma,
\end{split}
 \end{equation}
 which yields that
 \begin{equation}\label{eq:1-2-1} \int_{\r^{N}} U_{\varepsilon}^{p-2} \frac{\partial U_{\varepsilon}}{\partial x_{1}} v\, dx = 0.\end{equation}
Moreover, $\bar v_{k}$ is even in $x_{i}$, for $i =2, 3, \cdots, N$, so is $v$.

 Now, we claim that $v$ satisfies
 \[- \varepsilon^{2} \Delta v + v - (p-1)U_{\varepsilon}^{p-2} v = 0, \text{ in } \mathbb{R}^N. \]

For any $R>0$, let $\varphi \in C_0^\infty(B_R(0))$ be even in $x_i$, $i =2, 3, \cdots N$. Then $ \varphi(x- \xi_1) \in C_0^\infty(B_R(\xi_1))$ and  $\sum_{j=1}^k \mathcal{R}_{\frac{2\pi(j-1)}{k}}\varphi(x - \xi_1) \in H_s$, where $\mathcal{R}_{\theta}$ is the rotated operator given by
\[\mathcal{R}_{\theta} \varphi(x', x'') =\varphi(\mathcal{R}_\theta x', x'') =\varphi(x_{1}\cos\theta- x_{2}\sin\theta, x_{1}\sin\theta+x_{2}\cos\theta, x'')  \]
with $(x', x'') = \in \r^2 \times \r^{N-2}$. Since $\sum_{j=1}^k \mathcal{R}_{\frac{2\pi(j-1)}{k}}\varphi(x - \xi_1) \notin E_{k} $, we choose $a_{k}$ such that
\[\bar \varphi_{k}(x) =\sum_{j=1}^k \mathcal{R}_{\frac{2\pi(j-1)}{k}}\varphi(x - \xi_1)  - a_{k} \sum_{j=1}^{k}Z_{j} \in E_{k}.   \]
By the symmetry, one has
 \[a_{k} = \frac{\langle \sum_{j=1}^k \mathcal{R}_{\frac{2\pi(j-1)}{k}}\varphi(x - \xi_1), \sum_{j=1}^{k}Z_{j}  \rangle }{\| \sum_{j=1}^{k}Z_{j}\|^{2}} = \frac{k \langle \varphi(x - \xi_1), \sum_{j=1}^{k}Z_{j}  \rangle }{\| \sum_{j=1}^{k}Z_{j}\|^{2}}. \]
We may verify that
\[\|\sum_{j=1}^{k}Z_{j}\|^{2} = k \Bigl( \|\frac{\partial U_{\varepsilon} }{\partial x_{1}}\|^{2} + o(1)\Bigr). \]
Thus
\[|a_{k}| = \frac{|\langle \varphi(x - \xi_1), \sum_{j=1}^{k}Z_{j}  \rangle  |}{\|\frac{\partial U_{\varepsilon} }{\partial x_{1}}\|^{2} + o(1)} \leq C,\] which gives $\|\bar\varphi_{k}\|^{2} \leq Ck$.

Inserting $\bar\varphi_{k}$
into \eqref{eq:1.1-1}, we have
\begin{equation*}\begin{split}
o(\sqrt{k}) \|\bar\varphi_{k}\| = &  \int_{\r^{N}}\Bigl( \varepsilon^{2} \nabla v_{k}\nabla \bar\varphi_{k} + V(x) v_{k}\bar\varphi_{k} - (p-1) (\sum_{j=1}^{k}  U_{\varepsilon, \xi_{j}})^{p-2} v_{k}\bar\varphi_{k} \Bigr) \, dx - (p-1) \int_{\r^N} u_\varepsilon^{p-2} v_k \bar\varphi_{k} \, dx \\
& + O\Bigl(\sum_{j=1}^k \int_{\r^N} u_\varepsilon^{\frac{p-2}{2}} U_{\varepsilon, \xi_{j}}^{\frac{N-2}{2}} v_k \bar\varphi_{k} \, dx\Bigr)\\
= & \int_{\mathbb{R}^{N} \setminus B_{R_{0}}(0)} \Bigl( \varepsilon^{2} \nabla v_{k}\nabla \bar\varphi_{k} + V(x) v_{k}\bar\varphi_{k} - (p-1) (\sum_{j=1}^{k}  U_{\varepsilon, \xi_{j}})^{p-2} v_{k}\bar\varphi_{k} \Bigr) \, dx + O\bigl(k e^{-\frac{\sigma r}{\varepsilon}}\bigr)\\
= & k  \int_{\Omega_{1} \setminus B_{R_{0}}(0)} \Bigl( \varepsilon^{2} \nabla v_{k}\nabla \bar\varphi_{k} + V(x) v_{k}\bar\varphi_{k} - (p-1) (\sum_{j=1}^{k}  U_{\varepsilon, \xi_{j}})^{p-2} v_{k}\bar\varphi_{k} \Bigr) \, dx + O\bigl(k e^{-\frac{\sigma r}{\varepsilon}}\bigr),
\end{split}
\end{equation*}
which means
\[\begin{split} \int_{\Omega_1 \setminus B_{R_{0}}(0)}\bigl( \varepsilon^{2} \nabla v_{k}\nabla\bar \varphi_k+ V(x) v_{k}\bar \varphi_k - (p-1) (\sum_{i=1}^{k}  U_{\varepsilon, \xi_i})^{p-2} v_{k}\bar \varphi_k\bigr) \, dx  = & o(\frac{1}{\sqrt{k}})\|\bar\varphi_{k}\| + o_{k}(1)\\
=& o_{k}(1).
\end{split} \]
Then, for any $\varphi(x - \xi_{1})$, we have
\begin{align}\label{alpha}
& \int_{\Omega_1 \setminus B_{R_{0}}(0) }\bigl( \varepsilon^{2} \nabla v_{k}\nabla \varphi(x - \xi_{1})+ V(x) v_{k} \varphi(x - \xi_{1}) - (p-1) (\sum_{i=1}^{k}  U_{\varepsilon, \xi_i})^{p-2} v_{k} \varphi(x - \xi_{1})\bigr) \, dx  \nonumber \\
= & \frac{1}{k} \int_{\mathbb{R}^{N} \setminus B_{R_{0}}(0)} \bigl( \varepsilon^{2} \nabla v_{k}\nabla \sum_{j=1}^k \mathcal{R}_{\frac{2\pi(j-1)}{k}}\varphi(x - \xi_{1})+ V(x) v_{k} \sum_{j=1}^k \mathcal{R}_{\frac{2\pi(j-1)}{k}}\varphi(x - \xi_{1})   \nonumber \\
& - (p-1) (\sum_{i=1}^{k}  U_{\varepsilon, \xi_i})^{p-2} v_{k} \sum_{j=1}^k \mathcal{R}_{\frac{2\pi(j-1)}{k}} \varphi(x - \xi_{1}) \bigr) \, dx  \nonumber\\
= & \frac{1}{k}  \int_{\mathbb{R}^{N}} \bigl( \varepsilon^{2} \nabla v_{k}\nabla \bar\varphi_{k}+ V(x) v_{k}  \bar\varphi_{k}  - (p-1) (\sum_{i=1}^{k}  U_{\varepsilon, \xi_i})^{p-2} v_{k}  \bar\varphi_{k} - (p-1) u_{\varepsilon }^{p-2}v_{k} \bar\varphi_{k} \bigr) \, dx +o_{k}(1) \nonumber \\
& +\frac{a_{k}}{k}  \int_{\mathbb{R}^{N}} \bigl( \varepsilon^{2} \nabla v_{k}\nabla \sum_{j=1}^{k} Z_{j}+ V(x) v_{k}   \sum_{j=1}^{k} Z_{j}  - (p-1) (\sum_{i=1}^{k}  U_{\varepsilon, \xi_i})^{p-2} v_{k}   \sum_{j=1}^{k} Z_{j} \bigr) \, dx \nonumber \\
=& o_{k}(1)   +\alpha_{k} \langle \varphi(x- \xi_{1}), \sum_{j=1}^{k} Z_{j}\rangle,
 \end{align}
where
\[\alpha_{k} =  \| \sum_{j=1}^{k}Z_{j}\|^{-2} \int_{\mathbb{R}^{N}} \bigl( \varepsilon^{2} \nabla v_{k}\nabla \sum_{j=1}^{k} Z_{j}+ V(x) v_{k}   \sum_{j=1}^{k} Z_{j}  - (p-1) (\sum_{i=1}^{k}  U_{\varepsilon, \xi_i})^{p-2} v_{k}   \sum_{j=1}^{k} Z_{j} \bigr) \, dx.   \]

Next, we  estimate $\alpha_{k}$. Let $\eta \in C_{0}^{\infty}(B_{R}(\xi_{1}))$ be such that $\eta(x) = \eta(|x - \xi_{1}|)$, $\eta(x) =1$ in $B_{R/2}(\xi_{1})$, and $|\nabla \eta|\leq \frac{C}{R}$. Choosing $\varphi(x- \xi_{1}) = \eta \sum_{j=1}^{k}Z_{j}$ in \eqref{alpha}, we have
\[ \begin{split}& \alpha_{k} \langle\eta \sum_{j=1}^{k}Z_{j}, \sum_{j=1}^{k} Z_{j}\rangle \\
=&\int_{\Omega_1 \setminus B_{R_{0}}(0)}\bigl( \varepsilon^{2} \nabla v_{k}\nabla ( \eta \sum_{j=1}^{k}Z_{j})+ V(x) v_{k}  \eta \sum_{j=1}^{k}Z_{j} \\
& - (p-1) (\sum_{i=1}^{k}  U_{\varepsilon, \xi_i})^{p-2} v_{k} \eta \sum_{j=1}^{k}Z_{j}\bigr) \, dx + o_{k}(1)\\
= & \int_{\mathbb{R}^{N}} \Bigl( - \varepsilon^{2} \Delta ( \eta \sum_{j=1}^{k}Z_{j}) + V(x)   \eta \sum_{j=1}^{k}Z_{j} - (p-1) (\sum_{i=1}^{k}  U_{\varepsilon, \xi_i})^{p-2} \eta \sum_{j=1}^{k}Z_{j}\bigr) v_{k}\, dx + o_{k}(1)\\
= & o_{k}(1),
\end{split}\]
which means $\alpha_{k}=o_{k}(1)$. In views of \eqref{alpha}, there holds

\[ \int_{\Omega_1 \setminus B_{R_{0}}(0)}\bigl( \varepsilon^{2} \nabla v_{k}\nabla \varphi(x - \xi_{1})+ V(x) v_{k} \varphi(x - \xi_{1}) - (p-1) (\sum_{i=1}^{k}  U_{\varepsilon, \xi_i})^{p-2} v_{k} \varphi(x - \xi_{1})\bigr) \, dx = o_{k}(1).\]
namely,
\begin{equation*}
 \int_{\{ \Omega_1 \setminus B_{R_{0}}(0)- \xi_1\}}\bigl( \varepsilon^{2} \nabla  \bar v_{k}\nabla\varphi + V(x + \xi_1) \bar v_{k}\varphi - (p-1)   U_{\varepsilon}^{p-2} \bar v_{k}\varphi \bigr) \, dx = o_{k}(1),
\end{equation*}
for any $\varphi \in C_0^\infty(\r^N) $,which is even in $x_i$, $i=2, 3, \cdots, N$. Letting  $k \to \infty$,  we have
\begin{equation}\label{eq:1.1-3} \int_{\r^N} \bigl( \varepsilon^{2}  \nabla v \nabla\varphi +  v\varphi - (p-1)   U_{\varepsilon}^{p-2} v\varphi \bigr) \, dx = 0.  \end{equation}
 Since $v$ is even in $x_i$, $i=2, 3, \cdots, N$,  for any $\varphi \in C_0^\infty(\r^N)$, which is odd in $x_i$,   $i=2, 3, \cdots, N$, \eqref{eq:1.1-3} holds. Therefore,  \eqref{eq:1.1-3} holds for all $\varphi \in C_0^\infty(\r^N)$.

On the other hand, by the non-degeneracy of $U_{\varepsilon}$, it follows from \eqref{eq:1.1-3} that
\[v = \sum_{i=1}^{N}c_{i} \frac{\partial U_{\varepsilon}}{\partial x_i}.\]
Since $v$ is even in $x_{i}$, $i=2, 3, \cdots, N$, we have $v= c_{1} \frac{\partial U_{\varepsilon}}{\partial x_1}$.
By \eqref{eq:1-2-1},  we deduce that $c_{1} = 0$, that is, $v=0$, which leads to
\begin{equation}\label{eq:1.1-4}\int_{B_R(\xi_1)} v_k^2 \, dx = o(1). \end{equation}

Furthermore, the equation  $\|v_{k}\| = \sqrt{k}$ implies that  $\frac{1}{\sqrt{k}} v_{k}$ is bounded in $H^{1}(\r^{N})$. So we may assume that
\[\frac{1}{\sqrt{k}} v_{k} \rightharpoonup v_{0}, \text{ in } H^{1}(\r^{N}) \]
and
\[ \frac{1}{\sqrt{k}} v_{k} \to v_{0}, \text{ in } L_{loc}^{2}(\r^{N}). \]

Choosing $\psi \in C_{0}^{\infty}(\r^{N})$ with $\|\psi\| = 1$ in  \eqref{eq:1.1-1}, we have for small positive constant $\tau$ that
\begin{equation}\label{eq:1-3-1}\begin{split}
o(1) = & \int_{\r^{N}}\bigl( \varepsilon^{2} \nabla \bigl( \frac{1}{\sqrt{k}}  v_{k} \bigr)\nabla \psi + V(x)  \bigl( \frac{1}{\sqrt{k}}  v_{k} \bigr) \psi - (p-1) (\sum_{j=1}^{k}  U_{\varepsilon, \xi_j})^{p-2}  \bigl( \frac{1}{\sqrt{k}}  v_{k} \bigr) \psi \bigr) \, dx\\
& - (p-1) \int_{\r^N} u_\varepsilon^{p-2}  \bigl( \frac{1}{\sqrt{k}}  v_{k} \bigr) \psi \, dx  + O\Bigl(\sum_{i=1}^k \int_{\r^N} u_\varepsilon^{\frac{p-2}{2}} U_{\varepsilon, \xi_j}^{\frac{N-2}{2}}  \bigl( \frac{1}{\sqrt{k}}  v_{k} \bigr)  \psi\Bigr)\\
=   & \int_{\r^{N}}\bigl( \varepsilon^{2} \nabla \bigl( \frac{1}{\sqrt{k}}  v_{k} \bigr)\nabla \psi + V(x)  \bigl( \frac{1}{\sqrt{k}}  v_{k} \bigr) \psi  - (p-1) \int_{\r^N} u_\varepsilon^{p-2}  \bigl( \frac{1}{\sqrt{k}}  v_{k} \bigr) \psi \, dx  + O\Bigl(e^{-\frac{\tau r}{\varepsilon}}\Bigr).\\
\end{split}
\end{equation}
Let $k \to +\infty$, we find that $v_{0}$ satisfies
\[ - \varepsilon^{2} \Delta v_{0} + V(x) v_{0} - (p-1) u_{\varepsilon}^{p-2} v_{0} = 0.\]
It follows from the non-degeneracy of $u_{\varepsilon}$ in Lemma \ref{non} that $v_{0} = 0$. Thus,
\begin{equation}\label{eq:1-3-2}\int_{B_{R}(0)} v_{k}^{2} \, dx = o(k). \end{equation}
We deduce from \eqref{eq:1.1-4} and \eqref{eq:1-3-2} that
\begin{equation}\label{eq:1-3-3}\begin{split}
\int_{\r^{N}} W_{\varepsilon, k}^{p-2} v_{k}^{2} = & O\Bigl( \int_{\r^{N}} u_{\varepsilon}^{p-2} v_{k}^{2} \, dx \Bigr) + O\Bigl( \sum_{i=1}^{k}\int_{\r^{N}} U_{\varepsilon, \xi_j}^{p-2} v^{2}_{k} \, dx  \Bigr) \\
= & o(k),
\end{split}
\end{equation}
which contradicts with \eqref{eq:1.1-5}, that is,
\[\begin{split} o(k) = &   \int_{\r^{N}}\bigl( \varepsilon^{2} |\nabla v_{k}|^{2} + V(x) v_{k}^{2} - (p-1) W_{\varepsilon, k}^{p-2} v_{k}^{2} \bigr) dx\\
= & k + o(k).
\end{split} \]

\end{proof}

\bigskip

Next, we have the estimate for $l_{k}$.
\begin{lemma}\label{lm:2-2} For $r \in S_{k}$ and $k$ large enough,
\begin{equation}
\|l_{k}\| =  O\Bigl(  \frac{ k}{r^{m}}  \Bigr) + O \Bigl(k e^{-\frac{\tau r}{\varepsilon}} \Bigr) +  O \Bigl(k^{\frac{1}{2}} \Bigl( (r/k)^{-\frac{(N-1)}{2}} e^{-\frac{2 \pi r}{\varepsilon k }} \Bigr)^{ \min\{\frac{p-1}{2} -\tau, 1\}}\Bigr),
\end{equation}
where $\tau> 0$ is a small constant.
\end{lemma}

\begin{proof} For any $\varphi \in H^{1}(\r^{N})$ and a small enough constant $\delta >0$, by symmetry, we have
\begin{equation}\label{eq:1-4-1}\begin{split}
 \sum_{j=1}^{k}\int_{\r^{N}} \bigl(1 - V(x)\bigr) U_{\varepsilon, \xi_j} \varphi \, dx = & \sum_{j=1}^{k} \int_{\r^{N}} \bigl(1 - V(x+\xi_{j})\bigr) U_{\varepsilon} \varphi(x + \xi_{j}) \, dx\\
=  & O\Bigl(  \frac{ k}{r^{m}}  \Bigr)\|\varphi\|  = O\Bigl(    \frac{C}{k^{m-1} (\ln k)^{m}} \Bigr) \|\varphi\|.
 \end{split}
\end{equation}
On the other hand,  for small $\tau > 0$, we have
\begin{align*}\label{eq:1-4-2}
& | \int_{\r^{N}} \Bigl( W_{\varepsilon, k}^{p-1} - u_{\varepsilon}^{p-1} -\sum_{j=1}^{k} U_{\varepsilon, \xi_j}^{p-1} \Bigr) \varphi\, dx| \\
\leq & C\Bigl\{ \sum_{j=1}^{k} | \int_{\r^{N}} u_{\varepsilon}^{p-2} U_{\varepsilon, \xi_j} \varphi \, dx | +   | \int_{\r^{N}} \bigl( \sum_{j=1}^{k} U_{\varepsilon, \xi_j} \bigr)^{p-2} u_{\varepsilon} \varphi \, dx | \Bigr\} \\
&+  |\int_{\r^{N}} \bigl[ \bigl( \sum_{j=1}^{k} U_{\varepsilon, \xi_j} \bigr)^{p-1} - \sum_{j=1}^{k} U_{\varepsilon, \xi_j}^{p-1} \bigr] \varphi \, dx | \\
\leq & C k e^{-\frac{\tau r}{\varepsilon}} \|\varphi\|+  k|\int_{\Omega_{1}}  \bigl[ \bigl( \sum_{j=1}^{k} U_{\varepsilon, \xi_j} \bigr)^{p-1} - \sum_{j=1}^{k} U_{\varepsilon, \xi_j}^{p-1} \bigr] \varphi \, dx | \\
\leq & C k e^{-\frac{\tau r}{\varepsilon}} \|\varphi\| +C k  \sum_{j=2}^{k}\int_{\Omega_{1}} U_{\varepsilon, \xi_1}^{p-2}U_{\varepsilon, \xi_j} |\varphi| \, dx \\
\leq  &  C  k e^{-\frac{\tau r}{\varepsilon}}\|\varphi\|  + C k \sum_{j=2}^{k}U_{\varepsilon}^{\min\{\frac{p-1}{2}-\tau, 1\}}(|\xi_{j} - \xi_{1}|) \Bigl(\int_{\Omega_{1}} |\varphi|^{2}\Bigr)^{\frac{1}{2}} \\
\leq & C  k e^{-\frac{\tau r}{\varepsilon}}\|\varphi\|  + C k^{\frac{1}{2}} \Bigl( (r/k)^{-\frac{(N-1)}{2}} e^{-\frac{2 \pi r}{\varepsilon k }} \Bigr)^{ \min\{\frac{p-1}{2} -\tau, 1\}} \|\varphi\|.
\end{align*}
Thus,
\[\|l_{k}\| = O\Bigl(  \frac{ k}{r^{m}}  \Bigr) + O \Bigl(k e^{-\frac{(1-\tau) r}{\varepsilon}} \Bigr) +  O \Bigl(k^{\frac{1}{2}} \Bigl( (r/k)^{-\frac{(N-1)}{2}} e^{-\frac{2 \pi r}{\varepsilon k }} \Bigr)^{ \min\{\frac{p-1}{2} -\tau, 1\}}\Bigr).\]
\end{proof}

\bigskip

It is standard to  verify  the following result.

\begin{lemma}\label{lm:2-3-1} It holds
\[ \|R^{i}_{k}(\omega)\| = O\Bigl(\|\omega\|^{\min\{ p-1, 2\}-i}\Bigr), \text{ for } i =0, 1.  \]
\end{lemma}


\begin{lemma}\label{lm:2-3} There exists $k_{0}> 0$, for each $k > k_{0}$,  one has a map  $\omega_{k} \in E_{k}$ from $S_{k}$ to $H_{s}$ satisfying
\begin{equation}\label{eq:1-4-40} L_{k} \omega = l_{k} + R_{k}(\omega) + a_{k} \sum_{j=1}^{k}Z_{j}.
 \end{equation}
Moreover,
\[\|\omega_{k}\| \leq C \Bigl( \frac{ k}{r^{m}}   + k^{\frac{1}{2}} \Bigl( (r/k)^{-\frac{(N-1)}{2}} e^{-\frac{2 \pi r}{\varepsilon k }} \Bigr)^{ \min\{\frac{p-1}{2} -\tau, 1\}}\Bigr). \]
\end{lemma}

\begin{proof}Solving the equation \eqref{eq:1-4-40} is equivalent to solving the following equation
\begin{equation}\label{eq:1-4-41} P_{k}L_{k} \omega = P_{k}l_{k} + P_{k}R_{k}(\omega). \end{equation}
\eqref{eq:1-4-41} is well defined, since $u_{\varepsilon}$ is radially symmetric when $V(x)$ is radial, which implies that $P_{k}l_{k}$ and $P_{k}R_{k}(\omega)$ make sense.
By Lemma \ref{lm1}, $P_{k}L_{k}$ is invertible in $E_{k}$. Thus, \eqref{eq:1-4-41} can be rewritten as
\[\omega = A(\omega) : =(P_{k}L_{k}) ^{-1} P_{k} l_{k} +(P_{k}L_{k}) ^{-1} P_{k}R_{k}(\omega). \]
It follows from Lemma \ref{lm1} and Lemma \ref{lm:2-2} that
\[ \| (P_{k}L_{k}) ^{-1} P_{k} l_{k}  \| \leq  C \| P_{k}  l_{k}\| \leq C \|  l_{k}\|   \leq  C \Bigl( \frac{ k}{r^{m}}   + k^{\frac{1}{2}} \Bigl(  (r/k)^{-\frac{(N-1)}{2}} e^{-\frac{2 \pi r}{\varepsilon k }} \Bigr)^{ \min\{\frac{p-1}{2} -\tau, 1\}}\Bigr). \]

Next, we will apply the contraction mapping theorem to find the unique $\omega $ in a ball. Let
\[B:= \Bigl\{ \omega \in E_{k}:  \|\omega\| \leq C \Bigl( \frac{ k}{r^{m-\tau_{1}}}   + k^{\frac{1}{2}} \Bigl(  (r/k)^{-\frac{(N-1)}{2}} e^{-\frac{2 \pi r}{\varepsilon k }} \Bigr)^{ \min\{\frac{p-1}{2} -\tau, 1\}(1-\tau_{1})}\Bigr)\Bigr\}, \]
where $\tau_{1}> 0$ is a fixed small constant. Hence, it is sufficient to verify that  $A$ maps $B$ to $B$ and $A$ is a contraction map.

$(i)$ $A$ maps $B$ to $B$. In fact, for any $\omega \in B$, it follows from Lemmas \ref{lm:2-2} and \ref{lm:2-3-1} that
\begin{equation}\label{eq:1-4-43}\begin{split}
\|A\omega\| \leq & C \|P_{k}l_{k}\| + C \|P_{k}R_{k}(\omega)\|\\
	\leq & C \|l_{k}\| + C \|R_{k}(\omega)\| \\
\leq & C  \Bigl( \frac{ k}{r^{m}}   + k^{\frac{1}{2}} \Bigl( (r/k)^{-\frac{(N-1)}{2}} e^{-\frac{2 \pi r}{\varepsilon k }} \Bigr)^{ \min\{\frac{p}{2} -\tau, 1\}}\Bigr) + C \|\omega\|^{\min\{p-1, 2\}} \\
\leq & C \Bigl( \frac{ k}{r^{m-\tau_{1}}}   + k^{\frac{1}{2}} \Bigl( (r/k)^{-\frac{(N-1)}{2}} e^{-\frac{2 \pi r}{\varepsilon k }} \Bigr)^{ \min\{\frac{p}{2} -\tau, 1\}(1-\tau_{1})}\Bigr).
\end{split}
\end{equation}

$(ii)$ $A$ is a contraction map. For any $\omega_{1}, \omega_{2} \in B$, one has
\begin{equation}\label{eq:1-4-44}\begin{split}
\|A(\omega_{1}) - A(\omega_{2})\| = &  \|(P_{k}L_{k}) ^{-1}   P_{k} R_{k}(\omega_{1}) - (P_{k}L_{k}) ^{-1} P_{k}R_{k}(\omega_{2}) \|  \\
\leq & C \|R_{k}(\omega_{1}) -R_{k}(\omega_{2}) \| \\
\leq & C \|R^{'}(t\omega_{1} +(1-t) \omega_{2})\| \|\omega_{1} - \omega_{2}\|\\
\leq & C (\| \omega_{1}\| +\| \omega_{2}\|  )^{\min\{p-2, 1\}} \|\omega_{1} - \omega_{2}\|\\
\leq & \frac12  \|\omega_{1} - \omega_{2}\|.
\end{split}
\end{equation}
The assertion follows.

\end{proof}

\bigskip

\section{Energy estimate and the proof of Main Theorems }

\bigskip

In this section, we will give the proof of Theorem \ref{thm1}.
First, we estimate  $I(W_{\varepsilon,k})$.
Since
\[ W_{\varepsilon, k} = u_{\varepsilon} + \sum_{j=1}^{k}U_{\varepsilon, \xi_j},\]
we deduce
\begin{equation}\label{eq2.0}
\begin{split}
I(W_{\varepsilon, k}) = & I(u_{\varepsilon}) + k  \alpha +\sum_{j=1}^{k} \int_{\r^{N}} \Bigl( \varepsilon^{2} \nabla u_{\varepsilon} \nabla U_{\varepsilon, \xi_j} + V(x) u_{\varepsilon} U_{\varepsilon, \xi_j}\Bigr)\, dx + \frac12 \sum_{i, j =1}^{k} \int_{\r^{N}}  \bigl(V(x) - 1 \bigr)  U_{\varepsilon, \xi_i} U_{\varepsilon, \xi_j} \, dx  \\
& + \sum_{i<j} \int_{\r^{N}} \Bigl( \varepsilon^{2} \nabla U_{\varepsilon, \xi_i} \nabla U_{\varepsilon, \xi_j}+U_{\varepsilon, \xi_i} U_{\varepsilon, \xi_j} \Bigr)\, dx - \frac{1}{p} \int_{\r^{N}} \bigl[ |W_{\varepsilon,k}|^{p} - |u_{\varepsilon}|^{p} - \sum_{j=1}^{k} U_{\varepsilon, \xi_j}^{p}\bigr] \, dx   \\
= & I(u_{\varepsilon}) + k  \alpha +\sum_{j=1}^{k} \int_{\r^{N}} u_{\varepsilon}^{p-1} U_{\varepsilon, \xi_j} \, dx + \frac12 \sum_{i, j =1}^{k}\int_{\r^{N}}  \bigl(V(x) - 1 \bigr)  U_{\varepsilon, \xi_i} U_{\varepsilon, \xi_j} \, dx\\
& +   \sum_{i<j} \int_{\r^{N}}  U_{\varepsilon, \xi_i}^{p-1} U_{\varepsilon, \xi_j} \, dx   - \frac{1}{p} \int_{\r^{N}} \bigl[ |W_{\varepsilon,k}|^{p} - |u_{\varepsilon}|^{p} - \sum_{j=1}^{k} U_{\varepsilon, \xi_j}^{p}\bigr] \, dx \\
= &  I(u_{\varepsilon}) + k \alpha + \frac12 \sum_{i, j =1}^{k}\int_{\r^{N}}  \bigl(V(x) - 1\bigr)  U_{\varepsilon, \xi_i} U_{\varepsilon, \xi_j} \, dx \\
&  - \frac{1}{p} \int_{\r^{N}} \bigl[ |W_{\varepsilon,k}|^{p} - |u_{\varepsilon}|^{p} - \sum_{j=1}^{k} U_{\varepsilon, \xi_j}^{p} - p \sum_{j=1}^{k} u_{\varepsilon}^{p-1} U_{\varepsilon, \xi_j} - p  \sum_{i<j}  U_{\varepsilon, \xi_i}^{p-1}U_{\varepsilon, \xi_j} \bigr] \, dx,
\end{split}
\end{equation}
where $ \alpha= \bigl( \frac12 -\frac1p\bigr) \int_{\r^{N}} U_{\varepsilon}^{p}(x)dx$.

\bigskip
Now, we need to estimate each term in \eqref{eq2.0} separately.

\begin{lemma}\label{lm2-1}
It is true that
\[\sum_{i,j=1}^{k} \int_{\r^{N}} \bigl(V(x) - 1\bigr)  U_{\varepsilon, \xi_i}U_{\varepsilon, \xi_j} \,dx = \frac{ak\beta}{r^{m}} \Bigl(1  +O(r^{-\delta}) \Bigr),   \]
where $\beta = \int_{\r^{N}} U_{\varepsilon}^{2} \, dx $.
\end{lemma}

\begin{proof} Firstly, let us consider the case $i=j$.  Since \[V(|x|) = 1 + \frac{a}{|x|^{m}} + O(\frac{1}{|x|^{m+\delta}}), \, \, |x| \to \infty,\]
we have, for small $\theta> 0$,
\begin{equation}\label{eq2.1}\begin{split}
\int_{\r^{N}}  \bigl(V(x) - 1\bigr)  U^{2}_{\varepsilon, \xi_{j}} \, dx = & \int_{\r^{N}} \bigl(V(|x + \xi_{j}|) - 1\bigr)  U_{\varepsilon}^{2} \, dx \\
= & \int_{B_{\frac r2}(0)} \frac{a}{|x + \xi_{j}|^{m}} U_{\varepsilon}^{2} \, dx  + O\Bigl(  \int_{B_{\frac r2}(0)} \frac{1}{|x + \xi_{j}|^{m+\delta}}  U_{\varepsilon}^{2} \, dx \Bigr)+ O\Bigl( e^{-\frac{(1-\theta)r}{\varepsilon}}\Bigr).
\end{split}
\end{equation}
On the other hand, for $x \in B_{\frac{r}{2}(0)}$, there holds
\[\frac{1}{|x + \xi_{j}|^{m}} = \frac{1}{|\xi_{j}|^{m}}  \Bigl(1 + O\big(\frac{|x|}{|\xi_{j}|}\big)\Bigr),\]
which implies
\begin{equation}\label{eq2.2}
\int_{B_{\frac r2}(0)}  \frac{a}{|x + \xi_{j}|^{m}} U_{\varepsilon}^{2} \, dx =  \frac{a}{|\xi_{j}|^{m}} \int_{\r^{N}}U_{\varepsilon}^{2}\, dx + O\Bigl( r^{-(m+1)} + e^{-\frac{(1-\theta)r}{\varepsilon}}\Bigr).
\end{equation}
Similarly,
\begin{equation}\label{eq2.3}
O\Bigl( \int_{B_{\frac r2}(0)}  \frac{1}{|x + \xi_{j}|^{m+\delta}} U_{\varepsilon}^{2} \, dx \Bigr) = O\Bigl(  \frac{1}{|\xi_{j}|^{m+\delta}} \Bigr) = O\Bigl(\frac{1}{r^{m+\delta}} \Bigr).
\end{equation}
Next, if  $i\neq j$, by  symmetry, we see that
\begin{equation}\label{eq2.4}
\begin{split}
& \sum_{i\neq j}\int_{\r^{N}}\bigl(V(x) - 1\bigr)  U_{\varepsilon, \xi_{i}} U_{\varepsilon, \xi_{j}} \,dx \\
= & 2k \sum_{i=1}^{k} \sum_{i<  j} \int_{\Omega_{1} \setminus B_{R_{0}}(0) }\bigl(V(x) - 1\bigr)  U_{\varepsilon, \xi_{i}} U_{\varepsilon, \xi_{j}} \,dx +  \sum_{i\neq j}\int_{B_{R_{0}}(0)}\bigl(V(x) - 1\bigr)  U_{\varepsilon, \xi_{i}} U_{\varepsilon, \xi_{j}} \,dx  \\
= & 2k \sum_{j=2}^{k} \int_{\Omega_{1}}\bigl(V(x) - 1\bigr)  U_{\varepsilon, \xi_{1}} U_{\varepsilon, \xi_{j}} \,dx + 2k \sum_{i=2}^{k}  \sum_{i<  j} \int_{\Omega_{1}}\bigl(V(x) - 1\bigr)  U_{\varepsilon, \xi_{i}} U_{\varepsilon, \xi_{j}}\,dx + O\Bigl(e^{-\frac{(1-\theta)r}{\varepsilon }}\Bigr).
\end{split}
\end{equation}
If $x \in \Omega_{1}$, we have $|x - \xi_{j}| \geq |x - \xi_{1}|$. For any $\sigma \in (0, 1)$,
\[
U_{\varepsilon, \xi_{j}} \leq   C e^{-\frac{|x- \xi_{j}|}{\varepsilon}} \leq C e^{-\sigma \frac{|x - \xi_{j}|}{\varepsilon}} e^{-(1- \sigma) \frac{|x - \xi_{1}|}{\varepsilon}}
\leq  C e^{-\frac{\sigma}{2} \frac{|\xi_{1} - \xi_{j}|}{\varepsilon}} e^{-(1- \sigma) \frac{|x - \xi_{1}|}{\varepsilon}}.
\]
Hence,
\begin{equation}\label{sum} \begin{split}
\sum_{j =2}^{k} U_{\varepsilon, \xi_{j}} \leq  & C  e^{-(1- \sigma) \frac{|x - \xi_{1}|}{\varepsilon}}   \sum_{j =2}^{k}  e^{-\frac{\sigma}{2} \frac{|\xi_{1} - \xi_{j}|}{\varepsilon}} \\
\leq & C \sum_{j=2}^{k} e^{-(1- \sigma) \frac{|x - \xi_{1}|}{\varepsilon}}   e^{-\frac{\sigma r \sin \frac{(j-1)\pi}{k}}{ \varepsilon }} \leq  C_{1}e^{-(1- \sigma) \frac{|x - \xi_{1}|}{\varepsilon}}    e^{-\frac{\sigma \pi r  }{ k\varepsilon }} .
\end{split} \end{equation}
Consequently,
\begin{equation}\label{eq2.5}
\begin{split}
\Bigl|2k \sum_{j=2}^{k} \int_{\Omega_{1}}\bigl(V(x) - 1\bigr)  U_{\varepsilon, \xi_{1}} U_{\varepsilon, \xi_{j}} \,dx \Bigr| \leq &  2k C   e^{-\frac{\sigma \pi r  }{ k\varepsilon }} \int_{\Omega_{1}} \bigl|V(x) - 1\bigr|
e^{-\frac{|x-\xi_{1}|}{\varepsilon}}  e^{-(1- \sigma) \frac{|x - \xi_{1}|}{\varepsilon}} \\
= & O\Bigl(k  e^{-\frac{\sigma \pi r  }{ k\varepsilon }} r^{-m}  \Bigr)
\end{split}
\end{equation}
and
\begin{equation}\label{eq2.6}
\begin{split}
\Bigl| 2k \sum_{i=2}^{k}  \sum_{i<  j} \int_{\Omega_{1}}\bigl(V(x) - 1\bigr)  U_{\varepsilon, \xi_{i}} U_{\varepsilon, \xi_{j}} \,dx\Bigr| \leq & 2k  e^{-\frac{2\sigma \pi r  }{ k\varepsilon }}   \int_{\Omega_{1}}  \bigl|V(x) - 1\bigr| e^{-2(1- \sigma) \frac{|x - \xi_{1}|}{\varepsilon}} \\
= O\Bigl(k  e^{-\frac{2\sigma \pi r   }{ k\varepsilon }} r^{-m}  \Bigr).
\end{split}
\end{equation}
The result follows from \eqref{eq2.1}-\eqref{eq2.6}.

\end{proof}

\begin{lemma}\label{lm2-2} For $k$ sufficiently large, there holds 
\[\begin{split} &  \int_{\r^{N}} \bigl[ |W_{\varepsilon,k}|^{p} - |u_{\varepsilon}|^{p} - \sum_{j=1}^{k} U_{\varepsilon, \xi_{j}}^{p} - p \sum_{j=1}^{k} u_{\varepsilon}^{p-1} U_{\varepsilon, \xi_{j}} - p \sum_{i<j}  U_{\varepsilon, \xi_{i}}^{p-1} U_{\varepsilon, \xi_{j}} \bigr] \, dx \\
=&     \gamma k ( r/k)^{-\frac{N-1}{2}} e^{-\frac{2\pi r}{\varepsilon k}} \Bigl(1 + O(r^{-1}) \Bigr),
\end{split} \]
where  $\gamma$ \ is a positive constant.
\end{lemma}

\begin{proof} If $p \in (2, 3)$, we have
\begin{equation}\label{eq2.7}
\begin{split}
 & \int_{\r^{N}} \bigl[ |W_{\varepsilon,k}|^{p} - |u_{\varepsilon}|^{p} - \sum_{j=1}^{k} U_{\varepsilon, \xi_{j}}^{p} - p \sum_{j=1}^{k} u_{\varepsilon}^{p-1} U_{\varepsilon, \xi_{j}} - p \sum_{i<j}  U^{p-1}_{\varepsilon, \xi_{i}} U_{\varepsilon, \xi_{j}} \bigr] \, dx \\
 = & p   \sum_{i<j}  \int_{\r^{N}}  U^{p-1}_{\varepsilon, \xi_{i}} U_{\varepsilon, \xi_{j}} \, dx  + O\Bigl( \sum_{j=1}^{k} \int_{\r^{N}}  u_{\varepsilon}^{\frac{p}{2}} U_{\varepsilon, \xi_{j}}^{\frac{p}{2}} \, dx + \sum_{i< j} \int_{\r^{N}} U_{\varepsilon, \xi_{i}}^{\frac{p}{2}} U_{\varepsilon, \xi_{j}}^{\frac{p}{2}} \, dx   \Bigr).
 \end{split}
\end{equation}
We infer straightforwardly that
\begin{equation}\label{eq2.8}\begin{split}
 \sum_{i<j}  \int_{\r^{N}}  U^{p-1}_{\varepsilon, \xi_{i}} U_{\varepsilon, \xi_{j}} \, dx   = &\frac{ k}{2}  \sum_{j =2}^{k} \int_{\r^{N}} U^{p-1}_{\varepsilon, \xi_{1}} U_{\varepsilon, \xi_{j}} \, dx \\
= &  \frac{B_{1}k}{2}  \sum_{j =2}^{k} |\xi_{j} - \xi_{1}|^{-\frac{N-1}{2}} e^{-\frac{|\xi_{j}- \xi_{1}|}{\varepsilon}}  \Bigl(1 + O(r^{-1}) \Bigr)\\
=  &  \frac{B_{1}k}{2} |\xi_{2} - \xi_{1}|^{-\frac{N-1}{2}} e^{-|\xi_{2} -\xi_{1}|} \sum_{j=2}^{k} e^{-|\xi_{j}-\xi_{1}| + |\xi_{2} -\xi_{1}|} \frac{|\xi_{2}- \xi_{1}|^{\frac{N-2}{2}}}{|\xi_{j} -\xi_{1}|^{\frac{N-2}{2}}}\Bigl(1 + O(r^{-1}) \Bigr) \\
= &\gamma_{0} k \bigl(r \sin\frac{\pi}{k} \bigr)^{-\frac{N-1}{2}} e^{-\frac{2r\sin \frac{\pi}{k} }{\varepsilon}}\Bigl(1 + O(r^{-1}) \Bigr)\\
= & \gamma k ( r/k)^{-\frac{N-1}{2}} e^{-\frac{2\pi r}{\varepsilon k}}\Bigl(1 + O(r^{-1}) \Bigr) .
\end{split}
\end{equation}
The inequality $u_{\varepsilon}(x) \leq C e^{-\frac{(1-\delta)\sqrt{V(x_{0})} |x - x_{0}|}{\varepsilon}}$ with  $\delta > 0$ small yields that
\begin{equation}\label{eq2.9}
\begin{split}
\sum_{j=1}^{k} \int_{\r^{N}}  u_{\varepsilon}^{\frac{p}{2}} U_{\varepsilon, \xi_{j}}^{\frac{p}{2}} \, dx
= & O\Bigl(k \int_{\r^{N}} e^{-\frac{(1-\delta)p\sqrt{V(x_{0})}|x - x_{0}|}{2\varepsilon}} e^{- \frac{(1-\delta)p|x-\xi_{1}|}{2\varepsilon}}\Bigr)\\
=&  O\Bigl(k e^{-\frac{(1-\delta)p \min\{ \sqrt{V(x_{0})},1 \}|\xi_{1} - x_{0}|}{2\varepsilon}} \Bigr)
 \end{split}
\end{equation}
and
\begin{equation}\label{eq2.10}
\sum_{i< j} \int_{\r^{N}} U_{\varepsilon, \xi_{i}}^{\frac{p}{2}} U_{\varepsilon, \xi_{j}}^{\frac{p}{2}} \, dx  = O\Bigl(k \sum_{j=2}^{k} \int_{\Omega_{1}} U_{\varepsilon, \xi_{1}}^{\frac{p}{2}} U_{\varepsilon, \xi_{j}}^{\frac{p}{2}} \, dx \Bigr) =  O\Bigl(k e^{-\frac{(1-\delta)pr\sin \frac{\pi}{k}}{2\varepsilon}} \Bigr).
\end{equation}

If  $p \in [3, \frac{2N}{N-2} )$, we have
\begin{equation}\label{eq2.}
\begin{split}
 & \int_{\r^{N}} \bigl[ |W_{\varepsilon,k}|^{p} - |u_{\varepsilon}|^{p} - \sum_{j=1}^{k} U_{\varepsilon, \xi_{j}}^{p} - p \sum_{j=1}^{k} u_{\varepsilon}^{p-1} U_{\varepsilon, \xi_{j}} - p \sum_{i<j}  U^{p-1}_{\varepsilon, \xi_{i}} U_{\varepsilon, \xi_{j}} \bigr] \, dx \\
 = & p   \sum_{i<j}  \int_{\r^{N}}  U^{p-1}_{\varepsilon, \xi_{i}} U_{\varepsilon, \xi_{j}} \, dx  + O\Bigl( \sum_{j=1}^{k} \int_{\r^{N}}  u_{\varepsilon}^{p-2} U_{\varepsilon, \xi_{j}}^{2} \, dx + \sum_{i< j} \int_{\r^{N}} U_{\varepsilon, \xi_{i}}^{p-2} U_{\varepsilon, \xi_{j}}^{2} \, dx   \Bigr)\\
 =  &\gamma_{0} k\bigl(r \sin\frac{\pi}{k} \bigr)^{-\frac{N-1}{2}} e^{-\frac{2r\sin \frac{\pi}{k} }{\varepsilon}}  \Bigl(1 + O(r^{-1}) \Bigr)\\
 = &\gamma k ( r/k)^{-\frac{N-1}{2}} e^{-\frac{2\pi r}{\varepsilon k}}  \Bigl(1 + O(r^{-1}) \Bigr)  .
 \end{split}
\end{equation}

\end{proof}

\bigskip

\begin{proposition}\label{pp2} For $r \in S_{k}$, one has
\[\begin{split} I(W_{\varepsilon, k})
= & I(u_{\varepsilon}) + k  \alpha + \frac{ak\beta}{r^{m}} \Bigl(1  +O(r^{-\delta}) \Bigr) -  \frac{\gamma  k}{4} ( r/k)^{-\frac{N-1}{2}} e^{-\frac{2\pi r}{\varepsilon k}}  \Bigl(1 + O(r^{-1}) \Bigr).   \end{split} \]
\end{proposition}

\begin{proof}
The result follows directly from \eqref{eq2.0} and Lemmas~\ref{lm2-1}-\ref{lm2-2}.

\end{proof}

\begin{proof}[ {\bf Proof of Theorem \ref{thm1}}]

To demonstrate that $u_{k} = W_{\varepsilon, k} +\omega_{k}$ is the solution of the equation \eqref{main}, it suffices to verify that $a_{k} = 0$ in \eqref{eq:1-4-40}, which is equivalent to finding a critical point of $F(r) :=I(W_{\varepsilon, k} + \omega_{k})$.

For $r \in S_{k}$, by Lemmas \ref{lm:2-2} and \ref{lm:2-3}, we get
\begin{equation}\label{eq:m1}\begin{split}
F(r) = I(W_{\varepsilon, k} + \omega_{k})= &  I(W_{\varepsilon, k}) + \langle l
_{k}, \omega_{k}\rangle +O\Bigl( \|\omega_{k}\|^{2}\Bigr) \\
= &  I(u_{\varepsilon}) + k \Bigl\{  \alpha + \frac{a\beta}{r^{m}} \Bigl(1  +O(r^{-\delta}) \Bigr) -  \frac{\gamma  }{4} (r/k)^{-\frac{N-1}{2}} e^{-\frac{2 \pi r }{\varepsilon k}}  \Bigl(1 + O(r^{-1}) \Bigr) \Bigr\}.
\end{split}
\end{equation}
Define
\[F_{1}(r) =\frac{a \beta}{r^{m}} \Bigl(1  +O(r^{-\delta}) \Bigr) -  \frac{\gamma }{4} (r/k)^{-\frac{N-1}{2}} e^{-\frac{2 \pi r }{\varepsilon k}}  \Bigl(1 + O(r^{-1}) \Bigr). \]
Consider the following maximization problem
\[\max_{r\in S_{k}} F_{1}(r).\]
 Suppose that $r_{0}$ is a maximizer of $F_{1}(r)$, we will show that $r_{0}$ is an interior point of $S_{k}$.

 Define the function
 \[F_{11}(r) =\frac{a\beta}{r^{m}} - \frac{\gamma }{4} (r/k)^{-\frac{N-1}{2}} e^{-\frac{2 \pi r }{\varepsilon k}}  \]
 and it can be verified that $F_{11}(r)$
 has a maximum point
 \begin{equation}\label{r} \bar r =\bigl[ \frac{\varepsilon m}{2\pi} + o(1)\bigr] k\ln k \in S_{k},\end{equation}
 which satisfies
\[\frac{ma \beta}{{\bar r}^{m}} = \frac{\gamma }{4} (\bar r/k)^{-\frac{N-1}{2}}e^{-\frac{2\pi {\bar r} }{\varepsilon k}} \bigl[ \frac{N-1}{2} + \frac{2\pi {\bar r}}{\varepsilon k}\bigr].\]
 Moreover,
 \begin{equation}\label{f1}\begin{split} F_{1}(\bar r) =&  \frac{a \beta}{\bar r^{m}} \Bigl(1  +O(\bar r^{-\delta}) \Bigr) -  \frac{\gamma }{4} (\bar r/k)^{-\frac{N-1}{2}} e^{-\frac{2 \pi \bar r }{\varepsilon k}}  \Bigl(1 + O(\bar r^{-1}) \Bigr) \\
 = & \frac{\gamma }{4} (\bar r/k)^{-\frac{N-1}{2}} e^{-\frac{2 \pi \bar r }{\varepsilon k}} \Bigl[\frac{N-1}{2m} + \frac{2\pi \bar r}{m \varepsilon k}- 1\Bigr] + O\bigl(\bar r^{-(m +\delta)}\bigr)> 0.
 \end{split}\end{equation}
Taking $r=  [\frac{\varepsilon m}{2\pi} - \theta ] k\ln k$, one has, for $k$ large sufficiently,
  \begin{equation}\label{eq:m-3}\begin{split}
 & F_{1}\Bigl( [\frac{\varepsilon m}{2\pi} - \theta ] k\ln k\Bigr) \\
 = & \frac{a \beta}{\{ [\frac{\varepsilon m}{2\pi} - \theta ] k\ln k\}^{m}} \Bigl\{1 - \frac{\gamma }{4a \beta} \{[\frac{\varepsilon m}{2\pi} - \theta ] \ln k\}^{m -\frac{N-1}{2}} \cdot k^{  \frac{2\pi \theta}{\varepsilon}} \Bigr\} \cdot \Bigl( 1 + O(k^{-\delta})\Bigr)\\
  < &  0 < F_{1}(\bar r).
 \end{split}
 \end{equation}
 On the other hand,  by \eqref{r} and for $\theta_{1} > 0$ small, we have that
 \begin{equation}\label{eq:m-2}\begin{split}
 & F_{1}\Bigl([\frac{\varepsilon m}{2\pi} + \theta ] k\ln k\Bigr) \\
 = & \frac{a\beta}{\{ [\frac{\varepsilon m}{2\pi} + \theta ] k \ln k\}^{m}} \Bigl\{1 - \frac{\gamma }{4a \beta} \{[\frac{\varepsilon m}{2\pi} + \theta ]  \ln k \}^{m -\frac{N-1}{2}} \cdot k^{ - \frac{2\pi \theta}{\varepsilon}} \Bigr\} \cdot \Bigl( 1 + O(k^{-\delta})\Bigr)\\
 < &  \frac{a\beta}{\bar r^{m}} \Bigl\{1 - \frac{\gamma }{4a \beta} k^{- \theta_{1}} \Bigr\} \cdot \Bigl( 1 + O(k^{-\delta})\Bigr)\\
 < & \frac{\gamma }{4} \bar r^{-\frac{N-1}{2}} e^{-\frac{2 \pi \bar r }{\varepsilon k}} \Bigl[\frac{N-1}{2m} + \frac{2\pi \bar r}{m \varepsilon k}- 1\Bigr] + O\bigl(\bar r^{-(m +\delta)}\bigr) = F_{1}(\bar r).
 \end{split}
 \end{equation}
Hence, there exists an interior maximum point $r_{0}$ of $F_{1}$, that is, $W_{\varepsilon, k} + \omega_{k}$ is a solution of problem \eqref{main}. The proof is completed.

\end{proof}

\bigskip

Finally, we will give a sketch of the proof for Theorem \ref{thm2} since the approach is similar to that of  Theorem \ref{thm1}.

Set
\[\tilde W_{\varepsilon, k} = u_{\varepsilon} + \sum_{j=1}^{2k} (-1)^{j} U_{\varepsilon, \tilde \xi_{j}}. \]
Our goal is to look for a solution for \eqref{main} of the form
\[\tilde u_{\varepsilon, k} = \tilde{W}_{\varepsilon, k} + \tilde{\omega}_{\varepsilon, k},\]
where \[\tilde{\omega}_{\varepsilon, k} \in \tilde{E}_{k} := \Bigl\{ \omega \in \tilde{H}_{s}(\r^{N}): \sum_{j=1}^{2k} \int_{\r^{N}}  U^{p-2}_{\varepsilon, \tilde{\xi}_{j}} Z_{j} \varphi \, dx = 0\Bigr\} \]
and $\tilde{H}_{s}(\r^{N})$ is given by
\[\begin{split}\tilde{H}_{s}(\r^{N}) : =  \Bigl\{ u \in H^{1}(\r^{N}):  & u\ \mbox{is even in}\ x_i, i=2,\cdots,N \\
& u(r\cos\theta, r\sin\theta,x'')=
(-1)^{j}u\Bigl(r\cos(\theta+\frac{j\pi}{k}),r(\theta+\sin\frac{j\pi}{k}),x''\Bigr)\ \Bigr\}.
\end{split} \]

If $V(x)$ satisfies $(V_1)$ and $(V_2^{-})$, the finite-dimensional reduction result in Section~2 is still true.  Consequently, there exists $\tilde \omega_{\varepsilon, k} \in \tilde{E}_{k}$ such that, for any $r \in \tilde{S}_{k}$, $\tilde u_{\varepsilon, k}= \tilde{W}_{\varepsilon, k} + \tilde \omega_{\varepsilon, k}$ is a solution of
\[ - \Delta u + V(x)u = |u|^{p-2}u + \sum_{j=1}^{2k} U^{p-2}_{\varepsilon, \tilde{\xi}_{j}} Z_{j}, \text{ in } \r^{N}. \]
Moreover,
\[ \|\tilde \omega_{\varepsilon, k}\|  =  O\Bigl(  \frac{ k}{r^{m}}  \Bigr) + O \Bigl(k e^{-\frac{\tau r}{\varepsilon}} \Bigr) +  O \Bigl(k^{\frac{1}{2}} \Bigl( (r/k)^{-\frac{(N-1)}{2}} e^{-\frac{ \pi r}{\varepsilon k }} \Bigr)^{ \min\{\frac{p}{2} -\tau, 1\}}\Bigr). \]
In order to show that $\tilde u_{\varepsilon, k}$ is a solution of \eqref{main}, it is equivalent to proving that there exists a minimum point $\tilde{r}_{k} \in \tilde{S}_{k}$ of the following function
\[\tilde{F}(r) = I(\tilde{W}_{\varepsilon, k} + \tilde{\omega}_{\varepsilon, k}). \]
As the proof of Proposition \ref{pp2}, we have the following estimate
\begin{equation}\label{p2}\begin{split}
I(\tilde{W}_{\varepsilon, k} + \tilde{\omega}_{\varepsilon, k}) = &  I(u_{\varepsilon}) +2 k  \alpha - \frac{2 ak\beta}{r^{m}} \Bigl(1  +O( r^{-\delta}) \Bigr) +  \frac{\gamma  k}{2} (r/k)^{-\frac{N-1}{2}} e^{-\frac{\pi r}{\varepsilon k}}  \Bigl(1 + O(r^{-1}) \Bigr).
\end{split}
\end{equation}
To find the critical point $\tilde{r}_{k} \in \tilde{S}_{k}$ of the function
\[ \tilde{F}(r ) = I(u_{\varepsilon}) +2 k  \alpha - \frac{2 ak\beta}{r^{m}} \Bigl(1  +O(r^{-\delta}) \Bigr) +  \frac{\gamma  k}{2} (r/k)^{-\frac{N-1}{2}} e^{-\frac{\pi r}{\varepsilon k}}  \Bigl(1 + O(r^{-1}) \Bigr),  \]
we need only to verify that there is a minimum point $\tilde{r}_{k} \in \tilde{S}_{k}$ of
\[\tilde{F}_{1}(r) = - \frac{2 a \beta}{r^{m}} \Bigl(1  +O(r^{-\delta}) \Bigr) +  \frac{\gamma  }{2}  (r/k)^{-\frac{N-1}{2}} e^{-\frac{\pi r}{\varepsilon k}}  \Bigl(1 + O(r^{-1}) \Bigr).  \]
The result can be deduced as the proof of Theorem \ref{thm1}.
\bigskip

{\bf Statements and Declarations}  { There is no conflict of interest for all authors.}


\end{document}